\documentclass[a4paper,12pt]{amsart}
\usepackage[
hmarginratio={1:1},
vmarginratio={1:1},
textwidth=15.5cm,
textheight=21cm,
heightrounded
]{geometry}

\usepackage[utf8]{inputenc}
\usepackage{amsfonts,amstext,amsmath,amsthm,amscd,amssymb}
\usepackage[sc]{mathpazo}
\usepackage[dvipsnames]{xcolor}
\usepackage[pagebackref=false]{hyperref}
\usepackage{url}

\definecolor{dark-red}{rgb}{0.4,0.15,0.15}
\definecolor{dark-blue}{rgb}{0.15,0.15,0.4}
\definecolor{dark-green}{rgb}{0.15,0.4,0.15}
\hypersetup{
	colorlinks,
	linkcolor=dark-red,
	citecolor=dark-blue,
	urlcolor=dark-green
}

\numberwithin{equation}{section}

\theoremstyle{plain}
\newtheorem{theorem}{Theorem}[section]
\newtheorem{proposition}[theorem]{Proposition}
\newtheorem{lemma}[theorem]{Lemma}
\newtheorem{corollary}[theorem]{Corollary}

\theoremstyle{definition}
\newtheorem{definition}[theorem]{Definition}
\newtheorem{remark}[theorem]{Remark}
\newtheorem{example}[theorem]{Example}

\newcommand{\ZZ}{\mathbb{Z}}

\newcommand{\tr}{\operatorname{tr}}
\newcommand{\wgt}{\operatorname{wg}}
\newcommand{\sq}{\operatorname{sq}}
\newcommand{\sfpart}{\operatorname{SF}}
\newcommand{\supp}{\operatorname{supp}}
\newcommand{\onev}{\mathbf{1}}

\title[Squarefree Matrix Formulas for CWR]
{Squarefree Matrix Formulas for the $CWR$ Invariant\\ of Alternating Knots and Links}

\author{Micha{\l} Jab{\l}onowski}

\address{Institute of Mathematics, Faculty of Mathematics, Physics and
	Informatics, University of Gda\'nsk, 80-308 Gda\'nsk, Poland}

\email{\href{mailto:michal.jablonowski@gmail.com}
	{michal.jablonowski@gmail.com}}

\subjclass[2020]{57K10, 05C38, 05C50}
\keywords{alternating links, Tait graphs, CWR invariant, weighted adjacency matrices,
	squarefree extraction, simple cycles, log-determinant}

\date{\today}

\begin{document}
	
	\begin{abstract}
We give weighted-matrix formulas for the components of the $CWR$ invariant of
oriented non-split alternating links. After recalling the known trace formulas for
$CWR_{2}$ and $CWR_{3}$, we give a construction uniform in $k$: attaching an independent
commuting variable to each vertex of a consolidated Tait graph and extracting
the squarefree part of the resulting trace isolates simple cycles from closed
walks. This yields a formula for $CWR_k$ for every $k\ge 3$, a
log-determinant generating polynomial for each of the two Tait graphs, and an
equivalent M\"obius-inversion formula over principal submatrices. 

Specializing the uniform formula, we obtain explicit closed weighted formulas for $CWR_{4}$ and
$CWR_{5}$. We also record a bipartiteness
criterion for the vanishing of all odd components and a characteristic-polynomial
formula for the unweighted specialization of the first
nonvanishing odd component. The graph-theoretic
constructions apply to arbitrary finite simple loopless weighted graphs; the
alternating-link hypothesis enters through the invariance theorem for $CWR$.
\end{abstract}
	
	\maketitle
	%\tableofcontents
	
	\section{Introduction}
\label{sec:introduction}

The earlier $WRP$ invariant was introduced in \cite{JabWRP24} using the two
checkerboard (Tait) graphs of a reduced alternating diagram. After parallel
edges are consolidated, the resulting weighted graphs encode crossing data, and
weighted cycle sums produce the invariant. The $CWR$ invariant was subsequently
introduced in \cite{JabCWR24} as a refinement which separates these
contributions according to cycle length.

The first two terms, $CWR_2$ and $CWR_3$, admit simple trace formulas in terms
of weighted adjacency matrices \cite[Proposition~3.5]{JabCWR24}. The purpose
of the present paper is to continue this matrix approach. We give a
construction valid for every $k\ge3$: vertex variables
record the vertices visited by a closed walk, and squarefree extraction removes
exactly those walks which revisit a vertex. This gives a uniform trace formula
for the weighted simple-cycle sums defining $CWR_k$. Summing the trace formulas
over all lengths produces a log-determinant generating polynomial, and Boolean
M\"obius inversion gives an equivalent formula involving only traces of
principal submatrices. The specialization at $t=1$ recovers the original
$WRP$ invariant and gives a direct generating-function explanation of the
relation between $WRP$ and $CWR$ recorded in \cite[Sec.~3.1]{JabCWR24}.
Specializing the uniform formula to $k=4$ and $k=5$ then yields explicit
closed weighted formulas for $CWR_{4}$ and $CWR_{5}$.

The constructions in
Sections~\ref{sec:vertex-variables}--\ref{sec:low-order} are statements
about arbitrary finite simple loopless weighted graphs. Knot theory enters by
applying them to the two consolidated weighted Tait graphs and then invoking the
invariance theorem of \cite{JabCWR24}. Thus, the main issue is not the
existence of general simple-cycle counting methods, which is classical, but the
form in which those methods specialize to the weighted Tait graphs underlying
the $CWR$ sequence.

\subsection*{Relation to earlier work}

The mechanism used here has several important precedents. Nilpotent commuting
generators have been used to eliminate repeated vertices in the zeon-adjacency
work of Staples and Schott--Staples \cite{Sta08,SS08,SS11}. Giscard, Rochet
and Wilson \cite{GRW18} construct a Hopf algebra of self-avoiding hikes in
which simple cycles are irreducible elements and obtain them from a logarithm
with respect to induced-subgraph convolution. Giscard, Kriege and Wilson
\cite{GKW19} turn related induced-subgraph identities into an exact
simple-cycle and simple-path counting algorithm. Sinclair's squarefree
algebra formalism gives multilinear trace--log coefficients which, for graph
matrices, are cyclic products indexed by vertex subsets
\cite[Theorems~6.3 and~6.5]{Sin26}.

There is also a classical generating-function precedent for the
trace--determinant part of the construction. Flajolet and Sedgewick
\cite[Sec.~V.5, Proposition~V.6]{FS09} treat a graph with a formal weighted
adjacency matrix $G$ and show that powers of $G$ generate weighted walks,
while the entries of $(I-zG)^{-1}$ generate weighted walks (called paths there) of all lengths. For rooted
circuits they obtain a logarithmic-derivative formula involving
$\det(I-zG)$; see also \cite[Notes~V.26 and~V.28]{FS09}.
They further give a logarithmic determinant formula for unrooted circuits
\cite[Note~V.35]{FS09}.

These classical circuit formulas do not impose vertex self-avoidance: the
circuits being enumerated may revisit vertices. The additional step in the
present construction is the vertex marking by $X$ followed by squarefree
extraction, which removes precisely such repeated-vertex closed walks and
therefore leaves the weighted simple cycles occurring in $CWR$.

Accordingly, we do not claim the abstract principle that a logarithmic
determinant can isolate simple-cycle data as new. Our contribution is the
weighted Tait-graph realization tailored to $CWR$, the uniform squarefree
trace formula for all $k\ge3$, the principal-submatrix form, the explicit
formulas for $CWR_4$ and $CWR_5$ obtained from it, and the direct recovery of
$WRP$.

There is also a geometric and lattice-theoretic reason that cycle data of
Tait graphs are natural in the alternating setting. Greene
\cite{Gre17} showed that, for an alternating diagram, the
Gordon--Litherland pairing of a checkerboard surface is naturally
isometric, up to the appropriate sign, to the integral flow lattice of
the corresponding Tait graph. Under this identification the
irreducible elements of the flow lattice are precisely the oriented
cycles of the graph. Moreover, the discrete Torelli theorem for graphs
implies that the flow lattice determines a bridgeless graph up to
$2$-isomorphism; see \cite[Sec.~5]{Gre17}. Thus, the cycle structure
used in the definition of $CWR$ is closely related to the intrinsic
lattice structure carried by the checkerboard surfaces.

There is also an earlier knot-theoretic precedent for combining
cycle data with Whitney $2$-isomorphism. Murasugi and Przytycki
\cite[Sec.~3]{MP93} introduced the \emph{cycle index} $\alpha(G)$ of a
graph, defined through cyclically independent families of edges. They
proved that the cycle index is additive over the blocks of a graph and,
in particular,
$
G_{1}\text{ and }G_{2}\text{ are $2$-isomorphic}
\;\Longrightarrow\;
\alpha(G_{1})=\alpha(G_{2})
$
\cite[Proposition~3.7]{MP93}.

Their cycle index should not be confused with the data recorded by $CWR$.
The invariant $\alpha(G)$ is a single extremal integer measuring the
maximum size of a cyclically independent edge family, whereas $CWR$
retains the weighted contribution of every simple cycle and separates
these contributions according to their lengths. Nevertheless, the
result of Murasugi--Przytycki provides an early knot-theoretic precedent
for the stability of cycle-based graph information under
$2$-isomorphism.

There is a corresponding integral and matrix-theoretic formulation.
If $\partial$ is an oriented incidence matrix of a graph $G$, its
integral flow lattice is
$
\mathcal F(G)=\ker(\partial)\cap\mathbb Z^{E(G)}.
$
For an integral basis with basis vectors forming the columns of a matrix
$\iota$, Dochtermann, Meyers, Samavedam and Yi \cite{DMSY21} consider
the associated dual Laplacian
$
L^{*}=\iota^{\mathsf T}\iota.
$
For a plane graph, the boundaries of the bounded faces can be chosen as
a cycle basis, and with this choice $L^{*}$ is the reduced Laplacian of
the planar dual; see \cite[Propositions~3.7 and~4.2]{DMSY21}. Thus
planar duality connects cycles of one graph not only with the cut
structure of the dual, but also with a natural matrix attached to an
integral cycle basis.

This concerns a basis of the integral flow lattice, whereas $CWR$
uses the full collection of simple cycles rather than a chosen cycle
basis; the two types of data should therefore not be identified.

The present invariant contains additional information, namely the crossing weights
on the edges and the separation of cycle contributions according to
their lengths.

There is also a classical polynomial-theoretic precedent for passing from
plane-graph data to invariants of alternating links. Under the medial
construction, a plane graph $G$ gives an alternating link diagram for which
$G$ is one of the checkerboard graphs, and the classical Tutte--Jones
correspondence expresses the Jones polynomial of an alternating link in terms
of the Tutte polynomial of a Tait graph; see, for example, \cite{Prz06, Thi87}.
Thus, Tait graphs have long provided a bridge between graph invariants and
alternating-link invariants. The present construction uses the same
graph--link correspondence, but has a different enumerative target: rather than
summing over spanning subgraphs, $CWR$ records weighted simple cycles,
separated according to their lengths.

A still closer historical precedent for the simultaneous use of weights,
plane duality and knot diagrams appear in Przytycki's treatment of
\emph{chromatic graphs} \cite[Sec.~V.1.1]{Prz06}. There, signed and
edge-colored plane graphs are equipped with a deletion--contraction
polynomial, the dual graph carries the corresponding dual edge data, and
the resulting invariant satisfies a planar-duality relation
\cite[Lemma~V.1.20]{Prz06}. Przytycki also notes that this polynomial is a
$2$-isomorphism invariant for connected chromatic graphs.

The enumerative content is nevertheless different from that of $CWR$.
The chromatic-graph polynomial is a deletion--contraction state sum over
spanning subgraphs, whereas, after consolidation of parallel edges, $CWR$
records the products of the crossing weights around individual simple cycles
and separates these contributions according to cycle length.

At a more elementary graph-theoretic level, the numbers of simple cycles
of the various lengths form what is sometimes called the \emph{cycle vector}
of a graph; see, for example, \cite[Sec.~1.13.1]{Yad23}. Such a vector is
preserved by graph isomorphism. For $k\ge3$, the $CWR$ sequence may be
viewed as a weighted two-graph refinement of this idea: instead of retaining
only the number of $k$-cycles, it records the products of the edge weights
around them, separately for the two consolidated Tait graphs. The term
$CWR_2$ is treated separately and records the corresponding weighted edge
sum.

There is also a classical determinant-theoretic precursor to the appearance
of cycles in the present formulas. Sachs' expansion of the adjacency
characteristic polynomial expresses\\
$
\Phi_G(\lambda)=\det(\lambda I-A(G))
$
as
$
\Phi_G(\lambda)
=
\sum_{H\in\mathcal S(G)}
(-1)^{w(H)}2^{c(H)}
\lambda^{\,|V(G)|-|V(H)|},
$
where $w(H)$ denotes the number of connected components of $H$ and
$c(H)$ the number of its components that are cycles, $\mathcal S(G)$ is the family of Sachs subgraphs, whose connected
components are single edges or simple cycles;
see \cite[Theorem~5.2.9]{WW19}. Equivalent formulations in terms of
elementary subgraphs can be found in
\cite[Theorems~3.8 and~3.10]{Bap14} and
\cite[Theorem~8.2.3]{KK19}.

Thus, ordinary characteristic-polynomial coefficients already contain
cycle information, but in general, they combine it with contributions from
matchings and from disconnected unions of edges and cycles. The squarefree
trace--log construction used here has a different enumerative purpose:
it removes repeated-vertex closed walks and isolates individual weighted
simple cycles, while the variable $t$ retains their lengths. In this sense
$CWR$ separates precisely the cycle information that is amalgamated with
other Sachs-subgraph contributions in the ordinary characteristic polynomial.

There is also a classical Laplacian-matrix approach based directly on the two
checkerboard graphs of a link diagram. Lien and Watkins \cite{LW00} associate
signed plane graphs to the two checkerboard classes and prove that the signed
Laplacian matrices of a plane graph and its signed dual are Goeritz congruent.
Their argument is explicitly knot-theoretic: checkerboard graphs associated
with equivalent link diagrams have Laplacians related by Goeritz congruence,
and the duality of the two checkerboard graphs is reflected algebraically in
the corresponding quadratic forms. Thus, their work provides an earlier
matrix-theoretic use of essentially the same pair of planar graphs that
underlies $CWR$.

The present construction is different in two essential respects. We use
weighted adjacency matrices rather than Goeritz/Laplacian matrices, and the
quantities extracted from them are weighted simple-cycle sums separated by
length. Moreover, the signed-dual convention of \cite{LW00} assigns opposite
signs to corresponding primal and dual edges, whereas the variables $w,r$
used here record the crossing signs according to the convention of
\cite{JabCWR24}; consequently the signed Laplacians of \cite{LW00} should not
be identified directly with the weighted matrices used below.

A more recent knot-theoretic Laplacian construction is due to Silver and
Williams \cite{SW19}. Although its enumerative target is again different
from that of $CWR$, these constructions together illustrate the long-standing
role of matrices of checkerboard-associated graphs in encoding
knot-theoretic information.

The paper is organized as follows. Section~\ref{sec:cwr-background} recalls
$CWR$, $WRP$, and the known formulas for $CWR_2$ and $CWR_3$.
Sections~\ref{sec:vertex-variables} and~\ref{sec:logdet} develop the uniform
squarefree and log-determinant formulas, and
Section~\ref{sec:inclusion-exclusion} gives the equivalent
M\"obius-inversion form over principal submatrices.
Section~\ref{sec:low-order} specializes the uniform construction to the first
components beyond the previously known ones: it derives the explicit weighted
formulas for $CWR_4$ and $CWR_5$ and verifies them, together with the
M\"obius-inversion form, on an example. Finally,
Section~\ref{sec:odd-cycles} treats the odd components: it characterizes the
simultaneous vanishing of all of them by the bipartiteness of the corresponding
Tait graph, and expresses the first nonvanishing one through a coefficient of
the characteristic polynomial.

\section{The \texorpdfstring{$CWR$}{CWR} invariant and the previously known matrix formulas}
	\label{sec:cwr-background}
	
	We recall the definitions needed below from \cite{JabCWR24}.
	
	\begin{definition}
		A \emph{diagram} $D$ of a knot or link is a generic projection to the plane,
		viewed as a $4$-valent plane graph together with over--under information at
		each crossing. The diagram is \emph{alternating} if, while travelling along
		each component, over-crossings and under-crossings occur alternately. A
		diagram is \emph{reduced} if it has no nugatory crossings. In what follows we
		consider reduced alternating diagrams of non-split links and use the
		checkerboard-coloring convention fixed in \cite{JabCWR24}, in which a region of
		the type marked $A$ there is black.
	\end{definition}
	
	A \emph{region} of $D$ is a face of the underlying plane graph. A
	checkerboard coloring assigns the colors black and white to the regions so
	that regions sharing an edge have opposite colors. Associated with such a
	coloring are two planar dual graphs, the \emph{Tait graphs} $G_B$ and $G_W$.
	The vertices of $G_B$ correspond to black regions and the vertices of $G_W$
	correspond to white regions; each crossing of $D$ determines one edge in each
	of the two graphs.

This checkerboard construction has a classical signed-graph formulation.
Przytycki \cite[Lemma~V.1.17 and pp.~19--20]{Prz06} observes that the two
checkerboard colorings of a connected link diagram produce dual plane graphs:
duality interchanges the black--white edge attribute, while, for an oriented
diagram, the sign assigned to an edge is the sign of the corresponding
crossing. This distinction is useful in the present setting. The variables
$w,r$ used in $CWR$ encode the crossing sign, rather than the black--white
attribute coming from the local checkerboard configuration.

The planar duality between the two Tait graphs also has an algebraic
interpretation. For a connected plane multigraph $G$ with a plane dual
$G^{*}$, an edge set is the edge set of a cycle in $G$ if and only if
the corresponding dual edge set is a bond of $G^{*}$
\cite[Proposition~4.6.1]{Die17}. Equivalently, under the natural
identification of their edge sets, the cycle space of $G$ is the cut
space of $G^{*}$,
$
\mathcal C(G)=\mathcal B(G^{*});
$
see \cite[Proposition~4.6.2]{Die17}. Thus, before consolidation, the
cycle structure of either Tait graph is naturally dual to the cut
structure of the other.

It is important, however, to distinguish this classical plane duality
from the weighted simple graphs used in the definition of $CWR$.
The graphs $G_B^{*}$ and $G_W^{*}$ below are obtained by consolidating
parallel edges, and after this operation they need not form a dual pair
in the literal plane-multigraph sense. Accordingly, the cycle--bond
duality provides a structural background for the two Tait graphs, rather
than an identification of the two coordinates of $CWR$.

	Checkerboard graphs are among the oldest combinatorial encodings of knot
	diagrams and played an important role in the early tabulation of alternating
	knots; see, for example, the historical account in
	\cite[Chapter~11]{AKT21}.
	
	When $L$ has more than one component, we regard $L$ as an oriented link
	and fix an orientation of each component throughout. Crossing signs, and
	hence the edge weights used below, are taken with respect to these
	orientations. Accordingly, all statements concerning multi-component
	links are understood in the oriented category. For a knot, reversing its
	orientation does not change any crossing sign, so no orientation needs to
	be specified separately.
	
	Following \cite{JabCWR24}, every edge $e$ of $G_B$ and $G_W$ is assigned the
	weight
	$$
	\operatorname{wg}(e)=
	\begin{cases}
		w,&\text{if the corresponding crossing is positive},\\
		r,&\text{if the corresponding crossing is negative},
	\end{cases}
	$$
	with the crossing-sign convention used there.
	
	\begin{definition}
		The weighted simple graph $G_B^*$ is obtained from $G_B$ by consolidating all
		multiple edges joining the same pair of vertices into one edge whose weight is
		the product of the weights of the consolidated edges. The graph $G_W^*$ is
		defined analogously from $G_W$.
	\end{definition}
	
	Thus, if parallel edges $e_1,\ldots,e_s$ are consolidated to one edge $e$,
	then
	$
	\operatorname{wg}(e)=\prod_{j=1}^{s}\operatorname{wg}(e_j).
	$
	In particular, every edge weight in $G_B^*$ or $G_W^*$ is a monomial in
	$w$ and $r$.
	
	\begin{definition}
		For $i>2$, define
		$
		CB_i(w,r)=
		\sum_{C_B}\prod_{e\in E(C_B)}\operatorname{wg}(e),
		$
		where the sum ranges over all unoriented simple cycles $C_B$ of length $i$ in
		$G_B^*$. For $i=2$, define
		$
		CB_2(w,r)=\sum_{e\in E(G_B^*)}\operatorname{wg}(e).
		$
		Define $CW_i(w,r)$ analogously using $G_W^*$ in place of $G_B^*$.
		For every integer $i>1$, set
		$
		CWR_i(w,r)=\bigl(CB_i(w,r),CW_i(w,r)\bigr).
		$
		If there is no cycle of the required length, the corresponding polynomial is
		understood to be zero. Finally,
		$
		CWR(L)=\bigl(CWR_2,CWR_3,CWR_4,\ldots\bigr),
		$
		with trailing pairs $(0,0)$ omitted.
	\end{definition}
	
	The fact that this construction is independent of the chosen reduced
	alternating diagram is the basic invariance theorem proved in
	\cite[Theorem~2.1]{JabCWR24}.
	
	\begin{theorem}[CWR invariance \cite{JabCWR24}]
		\label{thm:cwr-invariance}
		Let $L$ be an oriented non-split alternating link, and let
		$D_1$ and $D_2$ be reduced alternating diagrams of $L$, with the
		component orientations induced from $L$. Then
		$
		CWR(D_1)=CWR(D_2).
		$
		Consequently, the sequence $CWR(L)$ and each component $CWR_i(L)$
		are invariants of the oriented non-split alternating link $L$.
		For knots, the orientation may be suppressed from the notation.
	\end{theorem}

\begin{remark}
	\label{rem:flypes}
	Theorem~\ref{thm:cwr-invariance} is invoked here in the form proved
	in \cite{JabCWR24}; the following discussion is only a geometric
	interpretation of the prime case and is not used as a proof of the
	full invariance statement.
	
	For a prime non-split alternating link, the diagrammatic background
	can be viewed particularly clearly through the flyping theorem.
	Menasco--Thistlethwaite proved that any two reduced alternating
	diagrams of such a link are related by flypes, and Kindred
	\cite{Kin22} subsequently gave an entirely geometric proof.
	In Kindred's formulation, a flype corresponds to an isotopy of one
	checkerboard surface together with a re-plumbing of the other
	\cite[Proposition~4.10 and Theorem~4.11]{Kin22}.
	
	On the Tait-graph side, the corresponding operation is closely
	related to Whitney $2$-isomorphism and hence to preservation of the
	underlying cycle matroid, whose circuits are precisely the cycles
	of the graph \cite[Sec.~10.11]{Yad23}. For an earlier
	knot-theoretic discussion of the corresponding two-vertex graph
	operation, described as mutation and related to the Whitney twist,
	see also \cite[Exercise~V.1.12]{Prz06}.
	
	Thus, in the prime case, the flyping theorem gives a geometric
	explanation for the stability of the cycle data entering $CWR$.
	For the general non-split case considered in
	Theorem~\ref{thm:cwr-invariance}, invariance is supplied by the
	cited theorem \cite{JabCWR24}, rather than being deduced here solely
	from the prime-case flyping statement.
\end{remark}
	
	\begin{remark}
		\label{rem:cr-writhe}
		When $L$ has more than one component, orientations are understood as fixed
		when crossing signs and writhe are used. Two numerical invariants are
		already visible in the first component. Since every reduced alternating
		diagram realizes the crossing number of the represented link, and since all
		reduced alternating diagrams of an oriented link have the same writhe, the
		formulas of \cite[Proposition~3.1]{JabCWR24} give
		$
		\operatorname{cr}(L)=
		\left.(\partial_{w}+\partial_{r})CB_2(w,r)\right|_{w=r=1},
		\;
		\operatorname{writhe}(L)=
		\left.(\partial_{w}-\partial_{r})CB_2(w,r)\right|_{w=r=1}.
		$
		In particular $\sum_{e\in E(G_B^{*})}\deg\operatorname{wg}(e)=\operatorname{cr}(L)$,
		and the same holds for $G_W^{*}$; we use this below as a consistency check on
		computed values.
	\end{remark}
	
	\begin{remark}
		\label{rem:signature}
		There is a further classical relation between the checkerboard data
		and the signature. With the standard checkerboard shading used by
		Traczyk \cite{Tra04}, a connected reduced alternating diagram satisfies
		$
		\sigma(L)
		=
		-\frac{1}{2}\operatorname{writhe}(D)
		+\frac{1}{2}(W-B),
		$
		where $B$ and $W$ denote the numbers of black and white regions,
		respectively; see \cite[Theorem~2]{Tra04}. Since consolidation of
		parallel edges does not change the vertex sets,
		$
		B=|V(G_B^*)|,
		\;
		W=|V(G_W^*)|.
		$
		Thus, the checkerboard graphs underlying $CWR$, together with the
		writhe already recovered from $CWR_2$ in
		Remark~\ref{rem:cr-writhe}, also contain the diagrammatic quantities
		appearing in this classical formula for the signature.
		
		If the black--white convention used here is opposite to Traczyk's
		standard shading convention, the roles of $B$ and $W$ in the displayed
		formula must of course be interchanged.
	\end{remark}

	The relation with the earlier invariant is particularly simple. In
	\cite{JabWRP24}, from $G_B^*$ and $G_W^*$ one forms directed graphs $G_B'$ and
	$G_W'$ by replacing every edge by two oppositely oriented edges with the same
	weight, and $WRP$ is the unordered pair obtained by summing the products of
	edge weights over directed cycles. Thus, an edge $e$ contributes the directed
	$2$-cycle of weight $\operatorname{wg}(e)^2$, while every unoriented simple
	cycle of length $k\ge3$ contributes two directed cycles, one in each
	orientation. Consequently
	\begin{equation}
		\label{eq:wrp-cwr}
		WRP(L)=
		\left\{
		CB_2(w^2,r^2)+2\sum_{k\ge3}CB_k(w,r),\;
		CW_2(w^2,r^2)+2\sum_{k\ge3}CW_k(w,r)
		\right\}.
	\end{equation}
	Equivalently, in the tuple notation of $CWR$,
	$$
	WRP(w,r)=
	\left\{CWR_2(w^2,r^2)+2\sum_{k>2}CWR_k(w,r)\right\},
	$$
	where the braces mean that the two coordinates are regarded as an unordered
	pair. This is the relation stated in \cite[Sec.~3.1]{JabCWR24}; here it also
	follows directly from the original definition in
	\cite[Secs.~2.1--2.2]{JabWRP24}. The convention that a doubled edge
	contributes a directed $2$-cycle is the one used there: for the trefoil,
	$CWR(K3a1)=\bigl((3w,w^{3}),(w^{3},0)\bigr)$ gives
	$\{3w^{2}+2w^{3},\,w^{6}\}$, which is the value
	$WRP(K3a1)=\{w^{6},2w^{3}+3w^{2}\}$ computed in \cite[Sec.~2.3]{JabWRP24}.

	Three structural properties of $CWR$ proved in \cite{JabCWR24} are useful
	context for the matrix formulas and provide independent consistency checks.
	First, $CWR$ is additive under connected sum:
	$
	CWR(K_1\#K_2)=CWR(K_1)+CWR(K_2)
	$
	for alternating knots, and likewise for the connected-sum operation on
	oriented non-split alternating links with the gluing data and
	component orientations fixed
	\cite[Theorem~3.2]{JabCWR24}. Second, mutant alternating knots have the same
	$CWR$ invariant \cite[Proposition~3.3]{JabCWR24}. This is compatible with the
	classical graph-theoretic role of Whitney $2$-isomorphism: Greene
	\cite[Sec.~5]{Gre17} explains the relation between the flow lattices of Tait
	graphs, their $2$-isomorphism classes, and mutation of reduced alternating
	diagrams. 
	
	The appearance of block decompositions in link invariants has a classical
	precedent in Murasugi--Przytycki \cite{MP93}; in particular, their cycle
	index is additive over the blocks of a graph
	\cite[Proposition~3.7]{MP93}. The $CWR$ additivity considered here has
	the same elementary graph-theoretic source---a simple cycle cannot pass
	through a cut vertex from one block to another---although the invariant
	being recorded is different.

	Third, if $mL$ denotes the
	mirror image of an alternating link and
	$CWR_i(L;w,r)=(CB_i(L;w,r),CW_i(L;w,r))$, then
	\begin{equation}
		\label{eq:mirror-cwr}
		CWR_i(mL;w,r)
		=
		\bigl(CW_i(L;r,w),\,CB_i(L;r,w)\bigr)
	\end{equation}
	by \cite[Proposition~3.4]{JabCWR24}. Thus, mirroring exchanges the two Tait
	graphs and simultaneously interchanges the crossing-weight variables $w$ and
	$r$.

	We next recall the matrix notation and the formulas for the first two terms.
	
	\begin{definition}
		Let $G$ be a finite simple graph without loops, with vertices
		$v_1,\ldots,v_n$ and edge weights in $\ZZ[w,r]$. Its \emph{weighted
			adjacency matrix} is the symmetric matrix
		$
		\overline A(G)=(\overline a_{ij})_{i,j=1}^{n},
		\;
		\overline a_{ij}=
		\begin{cases}
			\operatorname{wg}(e),&\text{if $v_i$ and $v_j$ are joined by an edge $e$},\\
			0,&\text{otherwise}.
		\end{cases}
		$
		The corresponding ordinary adjacency matrix $A(G)$ is obtained by replacing
		every nonzero edge weight by $1$.
	\end{definition}
	
	For the two consolidated Tait graphs we write
	$
	\overline A_B=\overline A(G_B^*),\;
	\overline A_W=\overline A(G_W^*),
	$
	and denote their ordinary adjacency matrices by $A_B$ and $A_W$,
	respectively.
	
	\begin{remark}
		\label{rem:vertex-order}
		The constructions below do not depend on the chosen ordering of the vertices.
		Indeed, under a relabelling represented by a permutation matrix $P$, the
		weighted adjacency matrix changes to
		$
		M'=PMP^{-1}.
		$
		If the vertex variables are relabelled simultaneously, then
		$X'=PXP^{-1}$, and consequently
		$
		(X'M')^k=P(XM)^kP^{-1}.
		$
		Hence, the traces occurring below are unchanged, as are the corresponding
		determinants. This is the weighted version of the standard permutation
		similarity of adjacency matrices; compare \cite[Theorem~2.1.6]{KK19}.
	\end{remark}
	
	\begin{proposition}[{\cite[Proposition~3.5]{JabCWR24}}]
		\label{prop:cwr23}
		For an oriented non-split alternating link $L$,
		$$
		CWR_2(L)=
		\left(
		\frac{\tr(\overline A_B A_B)}{2},
		\frac{\tr(\overline A_W A_W)}{2}
		\right),
		$$
		and
		$$
		CWR_3(L)=
		\left(
		\frac{\tr(\overline A_B^3)}{6},
		\frac{\tr(\overline A_W^3)}{6}
		\right).
		$$
	\end{proposition}
	
	\begin{remark}
		The factor $1/2$ in the formula for $CWR_2$ compensates for the two
		orientations in which an edge is traversed in the trace. For $CWR_3$, every
		unoriented triangle is represented by six closed walks: three choices of the
		starting vertex and two orientations. The unweighted case of the second
		formula is $c_{3}(G)=\tr(A^{3})/6$, which is \cite[Eq.~(1)]{HM71}.
	\end{remark}
%%%%%%%%%%%%

\section{Vertex variables and squarefree extraction}
\label{sec:vertex-variables}

The trace of a power of a weighted adjacency matrix records \emph{all}
closed walks of the given length, whereas the components of $CWR$ are sums
over \emph{simple} cycles. We organize the bookkeeping needed to remove the
repeated-vertex walks once and for all by attaching a variable to each vertex;
the repeated vertices are then detected by repeated variables, and the
correction terms become a single algebraic operation. This is the same basic
self-avoidance mechanism exploited by nilpotent (zeon) adjacency matrices in
\cite{SS08, SS11, Sta08}; our notation keeps the vertex variables explicit in
order to interface directly with the weighted $CWR$ polynomials. The explicit
corrections that this mechanism produces for $k=4$ and $k=5$ are carried out
in Section~\ref{sec:low-order}.

Throughout this section $G$ is a finite simple graph without loops with vertex
set $\{v_{1},\ldots,v_{n}\}$ and edge weights in $\ZZ[w,r]$, and
$M=(m_{ij})_{i,j=1}^{n}$ is its weighted adjacency matrix; thus $M$ is
symmetric with zero diagonal and $m_{ij}=0$ whenever $v_{i}v_{j}\notin E(G)$.
We write $[n]=\{1,\ldots,n\}$, and for a cycle $C$ in $G$ we put
$
\wgt(C)=\prod_{e\in E(C)}\wgt(e),
$
so that $CB_{i}=\sum_{C}\wgt(C)$, the sum being over the unoriented simple
cycles of length $i$ in $G_{B}^{*}$, and similarly for $CW_{i}$.

\begin{definition}
	\label{def:sq}
	Let $x_{1},\ldots,x_{n}$ be commuting indeterminates over
	$\ZZ[w,r]$ and put
	$
	X=\operatorname{diag}(x_{1},\ldots,x_{n}).
	$
	A monomial
	$x_{1}^{a_{1}}\cdots x_{n}^{a_{n}}$ is \emph{squarefree}
	(equivalently, \emph{multilinear}) if $a_i\le1$ for every $i$.
	
	For
	$
	P=\sum_{\alpha}c_{\alpha}x^{\alpha}
	\in \ZZ[w,r][x_{1},\ldots,x_{n}],
	$
	define its \emph{squarefree part} by
	$
	\sfpart(P)
	=
	\sum_{\alpha\in\{0,1\}^{n}}c_{\alpha}x^{\alpha}.
	$
	Thus $\sfpart(P)$ is obtained from $P$ by deleting every monomial
	divisible by $x_i^2$ for at least one $i$.
	We also define the $\ZZ[w,r]$-linear \emph{squarefree extraction}
	$
	\sq\colon
	\ZZ[w,r][x_{1},\ldots,x_{n}]
	\to
	\ZZ[w,r]
	$
	by
	$$
	\sq(P)
	=
	\sfpart(P)\big|_{x_{1}=\cdots=x_{n}=1}
	=
	\sum_{T\subseteq[n]}
	\Bigl[\prod_{i\in T}x_i\Bigr]P.
	$$
	Here $[\mu]P$ denotes the coefficient of the monomial $\mu$ in
	$P$. The constant monomial $1$ is regarded as squarefree. This
	convention is immaterial below, since the polynomials to which
	these operations are applied are homogeneous of positive degree
	in the variables $x_i$.
\end{definition}

The starting point is the classical adjacency-matrix interpretation of
walks: for an ordinary adjacency matrix $A$, the $(i,j)$-entry of $A^k$
counts the walks of length $k$ from $v_i$ to $v_j$; see, for example,
\cite[Theorem~2.3.4]{KK19} and \cite[Sec.~5.1]{WW19}. In particular,
diagonal entries count rooted closed walks. Equivalently, if
$\alpha_1,\ldots,\alpha_n$ are the adjacency eigenvalues of $G$, then
$
\tr(A^k)=\sum_{i=1}^{n}\alpha_i^k
$
is the number of closed walks of length $k$
\cite[Theorem~5.2.7]{WW19}. The construction below refines this classical
closed-walk count by recording the visited vertices and their edge weights,
so that repeated-vertex walks can subsequently be removed.

The following lemma is the weighted, vertex-marked version of
this standard observation: the edge weights record the weight of the walk,
while the variables $x_i$ record its visited vertices with multiplicity.

\begin{lemma}
	\label{lem:walk-expansion}
	For every integer $k\ge 1$,
	$$
	\tr\bigl((XM)^{k}\bigr)
	=\sum_{(i_{1},\ldots,i_{k})\in[n]^{k}}
	x_{i_{1}}x_{i_{2}}\cdots x_{i_{k}}\;
	m_{i_{1}i_{2}}m_{i_{2}i_{3}}\cdots m_{i_{k}i_{1}} .
	$$
	In particular $\tr((XM)^{k})$ is homogeneous of degree $k$ in
	$x_{1},\ldots,x_{n}$, and the summands with a nonzero coefficient are
	precisely the closed walks $v_{i_{1}}\to v_{i_{2}}\to\cdots\to
	v_{i_{k}}\to v_{i_{1}}$ of length $k$ in $G$, each counted with the product of
	the weights of its edges and with the product of the variables of the vertices
	it visits, multiplicities included.
\end{lemma}

\begin{proof}
	Since $X$ is diagonal, $(XM)_{ij}=x_{i}m_{ij}$, and expanding the $k$-fold
	matrix product gives
	$$
	\tr\bigl((XM)^{k}\bigr)
	=\sum_{i_{1},\ldots,i_{k}}(XM)_{i_{1}i_{2}}(XM)_{i_{2}i_{3}}\cdots
	(XM)_{i_{k}i_{1}}
	=\sum_{i_{1},\ldots,i_{k}}\Bigl(\prod_{j=1}^{k}x_{i_{j}}\Bigr)
	\prod_{j=1}^{k}m_{i_{j}i_{j+1}},
	$$
	with indices read cyclically, $i_{k+1}=i_{1}$. A summand is nonzero only if
	$m_{i_{j}i_{j+1}}\ne 0$ for all $j$, i.e. only if consecutive vertices are
	joined by an edge of $G$. Every factor contributes exactly one variable, so
	the total degree in the $x_{i}$ equals $k$.
\end{proof}

\begin{remark}
	In the unweighted specialization $M=A(G)$, the ordinary traces appearing
	before squarefree extraction are the spectral moments of the adjacency
	matrix:
	$
	\tr(A(G)^k)=\sum_{j=1}^{n}\alpha_j^k,
	$
	where $\alpha_1,\ldots,\alpha_n$ are the adjacency eigenvalues; equivalently,
	this is the number of closed walks of length $k$
	\cite[Theorem~5.2.7]{WW19}. Thus, the additional operation in
	Theorem~\ref{thm:sq-extraction} is precisely the passage from the
	classical closed-walk moment to its simple-cycle contribution.
\end{remark}

\begin{theorem}
	\label{thm:sq-extraction}
	Let $k\ge 3$ and let $\mathcal{C}_{k}(G)$ denote the set of
	unoriented simple cycles of length $k$ in $G$. Then
	$$
	\sfpart\!\left(\tr\bigl((XM)^k\bigr)\right)
	=
	2k\sum_{C\in\mathcal{C}_{k}(G)}
	\wgt(C)\prod_{v_i\in V(C)}x_i.
	$$
	Consequently, if
	$
	C_k(M)=\sum_{C\in\mathcal{C}_{k}(G)}\wgt(C)
	$
	denotes the total weight of the simple $k$-cycles, then
	$
	C_k(M)
	=
	\frac{1}{2k}\,
	\sq\!\left(\tr\bigl((XM)^k\bigr)\right).
	$
\end{theorem}

\begin{proof}
	By Lemma~\ref{lem:walk-expansion} the monomial in the $x_{i}$ attached to a
	closed walk $W=(i_{1},\ldots,i_{k})$ is $\prod_{j}x_{i_{j}}$, and this monomial
	is squarefree if and only if the indices $i_{1},\ldots,i_{k}$ are pairwise
	distinct. Hence,
	$
	\sfpart\!\left(\tr((XM)^k)\right)
	$
	is precisely the sum of the contributions of the closed $k$-walks
	visiting $k$ distinct vertices.
	
	Let $W=(i_{1},\ldots,i_{k})$ be such a walk. Its $k$ edges
	$v_{i_{1}}v_{i_{2}},\ldots,v_{i_{k}}v_{i_{1}}$ join $k$ distinct vertices in a
	closed chain, so the subgraph they span is a simple cycle $C$ of length $k$,
	and the weight recorded by $W$ is $\wgt(C)$ while the recorded monomial is
	$\prod_{v_i\in V(C)}x_i$. Conversely, a simple cycle $C\in\mathcal{C}_{k}(G)$
	arises in this way from exactly $2k$ closed walks: one chooses an initial
	vertex among the $k$ vertices of $C$ and one of the two orientations. For
	$k\ge 3$ these $2k$ walks are pairwise distinct as sequences. Indeed, two
	walks of the same orientation with distinct starting vertices differ already in
	their first entry, since the $k$ vertices of $C$ are distinct; and a walk and a
	reversed walk cannot coincide, because a common sequence would give a vertex
	$v$ of $C$ whose two neighbours along the walk, read forwards and backwards,
	agree, which for $k\ge3$ forces two distinct vertices of $C$ to be equal. This
	proves the first assertion. Applying $\sq$, equivalently evaluating the squarefree part at
	$x_{1}=\cdots=x_{n}=1$, gives $2k\,C_k(M)$, which proves the second
	assertion.
\end{proof}

\begin{corollary}
	\label{cor:cwrk-sq}
	Let $L$ be an oriented non-split alternating link, let $n_{B}$ and $n_{W}$ be the numbers of
	vertices of $G_{B}^{*}$ and $G_{W}^{*}$, and let $X_{B}$ and $X_{W}$ be
	diagonal matrices of independent vertex variables of the corresponding sizes.
	Then, for every $k\ge 3$,
	$$
	CWR_{k}(L)
	=\left(
	\frac{1}{2k}\,\sq\Bigl(\tr\bigl((X_{B}\overline A_{B})^{k}\bigr)\Bigr),\;
	\frac{1}{2k}\,\sq\Bigl(\tr\bigl((X_{W}\overline A_{W})^{k}\bigr)\Bigr)
	\right).
	$$
\end{corollary}

\begin{proof}
	Apply Theorem~\ref{thm:sq-extraction} to $M=\overline A_{B}$ and to
	$M=\overline A_{W}$, and use that $CB_{k}$ and $CW_{k}$ are by definition the
	total weights of the simple $k$-cycles of $G_{B}^{*}$ and $G_{W}^{*}$.
	Invariance is Theorem~\ref{thm:cwr-invariance}.
\end{proof}

\begin{remark}
	\label{rem:k12}
	The two exceptional indices behave as follows. For $k=1$ we get
	$\tr(XM)=0$, since $M$ has zero diagonal. For $k=2$ every closed $2$-walk
	$i\to j\to i$ already has distinct vertices, so no cancellation occurs and
	$
	\sq\Bigl(\tr\bigl((XM)^{2}\bigr)\Bigr)=\sum_{i\ne j}m_{ij}^{2}
	=2\sum_{e\in E(G)}\wgt(e)^{2}.
	$
	Thus, the case $k=2$ of the above mechanism records the \emph{squares} of the
	edge weights. Since every $\wgt(e)$ is a monomial $w^{a}r^{b}$, we have
	$\sum_{e}\wgt(e)^{2}=CB_{2}(w^{2},r^{2})$ for $G=G_{B}^{*}$, and likewise
	$\sum_{e}\wgt(e)^{2}=CW_{2}(w^{2},r^{2})$ for $G=G_{W}^{*}$. This is exactly
	the substitution occurring in the $WRP$ specialization recalled in
	Corollary~\ref{cor:wrp} below. Note also that every monomial occurring in
	$\sum_{e}\wgt(e)^{2}$ has even exponents, so the substitution
	$w^{2}\mapsto w$, $r^{2}\mapsto r$ is well defined on this polynomial and
	returns $CB_{2}(w,r)$. Hence, the $t^{2}$ coefficient of the generating
	polynomial of Theorem~\ref{thm:logdet-main} determines the second component
	$CB_{2}(w,r)$ itself, and not merely the substituted polynomial
	$CB_{2}(w^{2},r^{2})$.
\end{remark}

\section{The log-determinant generating function}
\label{sec:logdet}

We now assemble the polynomials of Theorem~\ref{thm:sq-extraction} for all $k$
into a single generating function. Two closely related constructions should be
kept in mind. In Sinclair's squarefree algebra, for $f_M(z)=\det(I+D(z)M)$ the
coefficients of $\log f_M$ are cyclic products indexed by vertex subsets
\cite[Theorem~6.3]{Sin26}. In the Hopf algebra of self-avoiding hikes of
Giscard, Rochet and Wilson, the logarithm of $\det(I-W)$ taken with respect to
the induced-subgraph convolution is exactly the formal series of simple cycles
\cite[Theorem~4.2]{GRW18}. Our formulation below introduces one additional
length variable $t$, specializes to symmetric weighted adjacency matrices, and
aggregates the multilinear coefficients by cardinality so that they are exactly
the components needed for $CWR$; the logarithm is the ordinary formal logarithm, followed by squarefree
extraction, and not a convolution logarithm.
Because the formal logarithm contains the coefficients $1/j$, the natural
ambient ring in this section is
$\mathbb{Q}[w,r][x_{1},\ldots,x_{n}][[t]]$. The final squarefree generating
polynomial will nevertheless have coefficients in $\ZZ[w,r]$.

\begin{definition}
	\label{def:lambda}
	For $M$ as above set
	$
	\Lambda_{M}(t,x)=-\log\det\bigl(I-tXM\bigr),
	$
	where $I$ is the identity matrix of size $n$ and $\log(1+u)=\sum_{j\ge
		1}(-1)^{j+1}u^{j}/j$. The definition makes sense because
	$\det(I-tXM)\in 1+t\,\ZZ[w,r][x][t]$, so that its logarithm is a well-defined
	element of $t\,\mathbb{Q}[w,r][x][[t]]$.
\end{definition}

\begin{lemma}
	\label{lem:logdet-trace}
	In $\mathbb{Q}[w,r][x][[t]]$ we have
	$$
	\Lambda_{M}(t,x)=\sum_{k\ge 1}\frac{t^{k}}{k}\,\tr\bigl((XM)^{k}\bigr).
	$$
\end{lemma}

\begin{proof}
	Write $N=XM$ and $f(t)=\det(I-tN)$. Since $f(0)=1$, both $f(t)$ and
	$I-tN$ are invertible as formal power series, with
	$(I-tN)^{-1}=\sum_{j\ge 0}t^{j}N^{j}$. We work over the
	$\mathbb{Q}$-algebra $\mathbb{Q}[w,r,x_1,\ldots,x_n][[t]]$. Jacobi's formula
	gives
	\begin{align*}
		\frac{d}{dt}\log\det(I-tN)
		&=\tr\Bigl((I-tN)^{-1}\frac{d}{dt}(I-tN)\Bigr)
		=-\tr\Bigl((I-tN)^{-1}N\Bigr)\\
		&=-\sum_{j\ge 0}t^{j}\tr\bigl(N^{j+1}\bigr).
	\end{align*}
	Hence, $\frac{d}{dt}\Lambda_{M}=\sum_{k\ge1}t^{k-1}\tr(N^{k})$. The series
	$\sum_{k\ge1}t^k\tr(N^k)/k$ has the same derivative and the same zero constant
	term, so the two series are equal.
\end{proof}

Extend $\sq$ first $\mathbb{Q}[w,r]$-linearly to
$\mathbb{Q}[w,r][x_1,\ldots,x_n]$, and then coefficientwise to power series in
$t$, that is,
$\sq\bigl(\sum_{k}P_{k}t^{k}\bigr)=\sum_{k}\sq(P_{k})\,t^{k}$.

\begin{theorem}
	\label{thm:logdet-main}
	Let $M$ be the weighted adjacency matrix of a finite simple loopless graph $G$
	on $n$ vertices, and let $C_{k}(M)$ be the total weight of the unoriented
	simple $k$-cycles of $G$. Then
	$$
	\sq\bigl(\Lambda_{M}(t,x)\bigr)
	=\sq\Bigl(-\log\det\bigl(I-tXM\bigr)\Bigr)
	=\Bigl(\sum_{e\in E(G)}\wgt(e)^{2}\Bigr)t^{2}
	+2\sum_{k\ge 3}C_{k}(M)\,t^{k}.
	$$
	In particular the coefficient of $t^{1}$ vanishes, which is why the right-hand
	side begins in degree $2$; and the right-hand side is a polynomial in $t$ of
	degree at most $n$.
\end{theorem}

\begin{proof}
	By Lemma~\ref{lem:logdet-trace} the coefficient of $t^{k}$ in
	$\Lambda_{M}(t,x)$ is $\tr((XM)^{k})/k$, so
	$
	\bigl[t^{k}\bigr]\,\sq\bigl(\Lambda_{M}\bigr)
	=\frac{1}{k}\,\sq\Bigl(\tr\bigl((XM)^{k}\bigr)\Bigr).
	$
	For $k=1$ this vanishes and for $k=2$ it equals $\sum_{e}\wgt(e)^{2}$ by
	Remark~\ref{rem:k12}. For $k\ge3$, Theorem~\ref{thm:sq-extraction} gives
	$\sq(\tr((XM)^{k}))=2k\,C_{k}(M)$, hence the coefficient of $t^{k}$ is
	$2C_{k}(M)$. Finally, $\tr((XM)^{k})$ is homogeneous of degree $k$ in the
	variables $x_{1},\ldots,x_{n}$ by Lemma~\ref{lem:walk-expansion}, while a
	squarefree monomial has degree at most $n$; therefore
	$\sq(\tr((XM)^{k}))=0$ for $k>n$.
\end{proof}

\begin{remark}
	\label{rem:hopf-comparison}
	The comparison with \cite[Theorem~4.2]{GRW18} is as follows. There one works
	in the trace monoid of hikes on a digraph, with a formal variable
	$\omega_{ij}$ attached to every directed edge, and the relevant logarithm
	$\log_{*}$ is taken with respect to the induced-subgraph convolution
	$(\varphi*\psi)[G]=\sum_{H\prec G}\varphi[H]\psi[G-H]$; the statement
	$-\log_{*}\det(I-W)=\Pi[G]$ then says that the simple cycles are the
	irreducible elements of a cocommutative Hopf algebra of self-avoiding hikes.
	In Theorem~\ref{thm:logdet-main} the logarithm is the ordinary formal
	logarithm, the self-avoidance is enforced by the vertex variables and the
	linear operator $\sq$ rather than by a convolution, and the extra variable $t$
	grades the output by cycle length, which is what makes the coefficients of the
	resulting polynomial the individual components $CB_{k}$ and $CW_{k}$ rather
	than an undifferentiated series. We therefore regard
	Theorem~\ref{thm:logdet-main} as a length-marked, weighted, Tait-graph form of
	a statement whose combinatorial content is already in \cite{GRW18} and, in the
	squarefree-algebra formulation closest to ours, in \cite[Theorem~6.3]{Sin26}.
\end{remark}

\begin{corollary}
	\label{cor:cwr-logdet}
	For an oriented non-split alternating link $L$ and every $k\ge 3$,
	$$
	CWR_{k}(L)=
	\frac{1}{2}
	\Bigl(
	\bigl[t^{k}\bigr]\,\sq\bigl(-\log\det(I-tX_{B}\overline A_{B})\bigr),
	\;
	\bigl[t^{k}\bigr]\,\sq\bigl(-\log\det(I-tX_{W}\overline A_{W})\bigr)
	\Bigr),
	$$
	with $X_{B}$, $X_{W}$ as in Corollary~\ref{cor:cwrk-sq}. Moreover the pair of
	polynomials
	$\sq(-\log\det(I-tX_{B}\overline A_{B}))$ and
	$\sq(-\log\det(I-tX_{W}\overline A_{W}))$ determines the whole invariant
	$CWR(L)$: the components with $k\ge3$ are read off as above, and $CWR_{2}$ is
	recovered from the two coefficients of $t^{2}$ by the substitution
	$w^{2}\mapsto w$, $r^{2}\mapsto r$ of Remark~\ref{rem:k12}.
\end{corollary}

\begin{proof}
	Immediate from Theorem~\ref{thm:logdet-main} applied to $\overline A_{B}$ and
	$\overline A_{W}$, whose simple $k$-cycle weights are $CB_{k}$ and $CW_{k}$,
	together with Remark~\ref{rem:k12} for the last assertion.
\end{proof}

\begin{corollary}[$WRP$ as the value at $t=1$]
	\label{cor:wrp}
	With the notation of \cite{JabWRP24} and \cite[Sec.~3.1]{JabCWR24},
	$
	\sq\Bigl(-\log\det\bigl(I-tX_{B}\overline A_{B}\bigr)\Bigr)\Big|_{t=1}
	=CB_{2}(w^{2},r^{2})+2\sum_{i>2}CB_{i}(w,r),
	$
	and the same identity holds with $B$ replaced by $W$. Hence, the unordered
	pair of these two evaluations is exactly the invariant $WRP(w,r)$ introduced
	in \cite{JabWRP24}.
\end{corollary}

\begin{proof}
	By Theorem~\ref{thm:logdet-main} the left-hand side is a polynomial in $t$, so
	the substitution $t=1$ is legitimate and produces
	$\sum_{e}\wgt(e)^{2}+2\sum_{k>2}CB_{k}(w,r)$. Every edge weight of
	$G_{B}^{*}$ is a monomial $w^{a}r^{b}$, so
	$\sum_{e}\wgt(e)^{2}=CB_{2}(w^{2},r^{2})$ as in Remark~\ref{rem:k12}. This is
	precisely the black coordinate of
	$CWR_{2}(w^{2},r^{2})+2\sum_{i>2}CWR_{i}(w,r)$, and the white coordinate is
	obtained in the same way. By Equation~\eqref{eq:wrp-cwr}, which is the
	directed-cycle definition of \cite{JabWRP24} rewritten in terms of $CWR$, the
	unordered pair of these two coordinates is $WRP(w,r)$.
\end{proof}

\begin{remark}
	\label{rem:role-of-x}
	Setting $x_{1}=\cdots=x_{n}=1$ \emph{before} extracting the squarefree part
	gives the classical trace--determinant generating series
	$$
	-\log\det(I-tM)
	=
	\sum_{k\ge1}\frac{t^{k}}{k}\tr(M^{k}),
	$$
	which records all closed walks; compare
	\cite[Proposition~V.6 and Notes~V.26, V.28]{FS09}.
	The closed $k$-walks are recorded with coefficient $1/k$.
	The vertex variables and the operator $\sq$ are exactly what removes the
	walks that revisit a vertex, so that only the simple cycles, i.e.\ the cycles
	appearing in the definition of $CWR$, survive.
	
	In this sense Theorem~\ref{thm:logdet-main} is the generating-function form of
	the inclusion--exclusion that is carried out by hand for $k=4$ and $k=5$ in
	Section~\ref{sec:low-order}.
\end{remark}

\section{An inclusion--exclusion form of the extraction}
\label{sec:inclusion-exclusion}

Theorem~\ref{thm:logdet-main} is stated in terms of the operator $\sq$; we now
make that operator explicit by a M\"obius inversion over the Boolean lattice of
vertex subsets. This yields a formula for $C_{k}(M)$ involving only ordinary
traces (equivalently, only determinants) of principal submatrices.

\begin{definition}
	For $S\subseteq[n]$ let $M[S]=(m_{ij})_{i,j\in S}$ be the principal submatrix
	of $M$ on the rows and columns indexed by $S$; it is the weighted adjacency
	matrix of the induced subgraph $G[S]$. For $S=\varnothing$ we use the
	convention $\tr(M[S]^k)=0$ for $k\ge1$. Let
	$\onev_{S}\in\{0,1\}^{n}$ be the indicator vector of $S$.
\end{definition}

\begin{lemma}
	\label{lem:restriction}
	For every $S\subseteq[n]$ and every $k\ge1$,
	$
	\tr\bigl((XM)^{k}\bigr)\Big|_{x=\onev_{S}}=\tr\bigl(M[S]^{k}\bigr).
	$
\end{lemma}

\begin{proof}
	By Lemma~\ref{lem:walk-expansion} the substitution $x_{i}=1$ for $i\in S$ and
	$x_{i}=0$ otherwise kills exactly the closed walks visiting a vertex outside
	$S$ and leaves the remaining ones unchanged. The surviving walks are the
	closed $k$-walks of $G[S]$, whose generating sum is $\tr(M[S]^{k})$.
\end{proof}

\begin{lemma}
	\label{lem:mobius}
	Let $P\in\ZZ[w,r][x_{1},\ldots,x_{n}]$ be homogeneous of degree $k$ in
	$x_{1},\ldots,x_{n}$, and let $T\subseteq[n]$ with $|T|=k$. Then
	$
	\Bigl[\prod_{i\in T}x_{i}\Bigr]P
	=\sum_{S\subseteq T}(-1)^{k-|S|}\,P(\onev_{S}).
	$
\end{lemma}

\begin{proof}
	Write $P=\sum_{\mu}c_{\mu}\mu$ over monomials $\mu$ in the $x_{i}$, and let
	$\supp(\mu)$ be the set of indices occurring in $\mu$. Then
	$\mu(\onev_{S})=1$ if $\supp(\mu)\subseteq S$ and $\mu(\onev_{S})=0$
	otherwise, so
	$$
	\sum_{S\subseteq T}(-1)^{|T|-|S|}P(\onev_{S})
	=\sum_{\mu}c_{\mu}\!\!\sum_{\supp(\mu)\subseteq S\subseteq T}\!\!(-1)^{|T|-|S|},
	$$
	where $\mu$ runs over the monomials with $\supp(\mu)\subseteq T$. The inner
	sum equals $\sum_{j=0}^{d}\binom{d}{j}(-1)^{d-j}=0$ with
	$d=|T|-|\supp(\mu)|$, unless $d=0$, i.e. unless $\supp(\mu)=T$, in which case
	it equals $1$. Since $P$ is homogeneous of degree $k=|T|$, the only monomial
	with support equal to $T$ is $\prod_{i\in T}x_{i}$, and the claim follows.
\end{proof}

\begin{theorem}
	\label{thm:incl-excl}
	Let $M$ be as above and let $k\ge3$. Then
	$$
	C_{k}(M)
	=\frac{1}{2k}\sum_{\substack{T\subseteq[n]\\ |T|=k}}\ \sum_{S\subseteq T}
	(-1)^{k-|S|}\tr\bigl(M[S]^{k}\bigr)
	=\frac{1}{2k}\sum_{\substack{S\subseteq[n]\\ |S|\le k}}
	(-1)^{k-|S|}\binom{n-|S|}{k-|S|}\tr\bigl(M[S]^{k}\bigr),
	$$
	with the convention $\binom{a}{b}=0$ for $b<0$ or $b>a$. Consequently, for an oriented non-split alternating link $L$ and $k\ge3$,
	$
	CWR_{k}(L)=\bigl(C_{k}(\overline A_{B}),\,C_{k}(\overline A_{W})\bigr)
	$
	with $C_{k}$ given by the above expression.
\end{theorem}

\begin{proof}
	By Lemma~\ref{lem:walk-expansion}, $P=\tr((XM)^{k})$ is homogeneous of degree
	$k$, so by Definition~\ref{def:sq} only the squarefree monomials of degree $k$
	contribute to $\sq(P)$, that is
	$$
	\sq(P)=\sum_{\substack{T\subseteq[n]\\|T|=k}}\Bigl[\prod_{i\in T}x_{i}\Bigr]P .
	$$
	Lemma~\ref{lem:mobius} rewrites each inner coefficient as
	$\sum_{S\subseteq T}(-1)^{k-|S|}P(\onev_{S})$, and
	Lemma~\ref{lem:restriction} identifies $P(\onev_{S})$ with $\tr(M[S]^{k})$.
	Theorem~\ref{thm:sq-extraction} then gives the first equality. For the second,
	exchange the order of summation: a fixed $S$ with $|S|=j$ is contained in
	exactly $\binom{n-j}{k-j}$ subsets $T\subseteq[n]$ of cardinality $k$.
\end{proof}

\begin{corollary}
	\label{cor:incl-excl-det}
	For $k\ge3$,
	$$
	C_{k}(M)=\frac{1}{2}\sum_{\substack{S\subseteq[n]\\ |S|\le k}}
	(-1)^{k-|S|}\binom{n-|S|}{k-|S|}
	\bigl[t^{k}\bigr]\Bigl(-\log\det\bigl(I-tM[S]\bigr)\Bigr).
	$$
\end{corollary}

\begin{proof}
	Apply Lemma~\ref{lem:logdet-trace} with $X=I$ to the matrix $M[S]$: the
	coefficient of $t^{k}$ in $-\log\det(I-tM[S])$ equals $\tr(M[S]^{k})/k$. Now
	substitute into Theorem~\ref{thm:incl-excl}.
\end{proof}

\begin{remark}
	\label{rem:hamiltonian}
	For a fixed $T$ with $|T|=k$ the inner sum of Theorem~\ref{thm:incl-excl},
	$
	\gamma(T)=\sum_{S\subseteq T}(-1)^{k-|S|}\tr\bigl(M[S]^{k}\bigr),
	$
	equals $2k$ times the total weight of the Hamiltonian cycles of the induced
	subgraph $G[T]$; indeed, by the proof above, $\gamma(T)$ collects exactly the
	closed $k$-walks visiting all $k$ vertices of $T$. Thus
	Theorem~\ref{thm:incl-excl} is the weighted form of the classical
	inclusion--exclusion count of Hamiltonian cycles, applied simultaneously to all
	$k$-element vertex subsets. Computationally, if Corollary~\ref{cor:incl-excl-det} is evaluated through
	the determinant, it requires one determinant (equivalently, one relevant
	characteristic-polynomial) computation per subset, while
	Theorem~\ref{thm:sq-extraction} may instead be
	implemented by multiplying matrices over the quotient ring
	$\ZZ[w,r][x_{1},\ldots,x_{n}]/(x_{1}^{2},\ldots,x_{n}^{2})$, in which
	non-squarefree monomials vanish automatically; this is the nilpotent adjacency
	matrix of \cite{Sta08}.
\end{remark}

\begin{remark}
	\label{rem:giscard-comparison}
	Theorem~\ref{thm:incl-excl} should be compared with the sharper
	induced-subgraph formula of Giscard, Kriege and Wilson. Writing $\gamma(\ell)$
	for the number of simple $\ell$-cycles of a digraph $G$, $H\prec_{\mathrm{conn}}G$
	for the weakly connected induced subgraphs and $N(H)$ for the set of neighbours
	of $H$ in $G$, their Equation~(2) reads
	$$
	\gamma(\ell)=\frac{(-1)^{\ell}}{\ell}
	\sum_{H\prec_{\mathrm{conn}}G}\binom{|N(H)|}{\ell-|H|}(-1)^{|H|}
	\tr\bigl(A_{H}^{\ell}\bigr),
	$$
	and is proved in \cite{GRW18} from the Hopf algebra recalled in
	Remark~\ref{rem:hopf-comparison}; see \cite[Eq.~(2)]{GKW19}. Two differences
	matter in practice. Their sum runs only over \emph{connected} induced
	subgraphs, and the binomial coefficient $\binom{|N(H)|}{\ell-|H|}$ vanishes
	unless $|H|\le\ell\le|H|+|N(H)|$, so that only subgraphs on at most $\ell$
	vertices contribute; ours runs over all subsets of size at most $k$ with the
	coefficient $\binom{n-|S|}{k-|S|}$. On many sparse graphs there are far fewer connected induced subgraphs
	than subsets, so their form can be substantially preferable
	computationally. 
	
	We keep Theorem~\ref{thm:incl-excl} because it is the
	statement that comes out of the operator $\sq$ by pure M\"obius inversion, and
	because in the weighted setting each summand is a polynomial in $\ZZ[w,r]$
	which is directly a partial contribution to $CB_{k}$ or $CW_{k}$; but no claim
	of computational novelty is intended.
\end{remark}
%%%%%%%%%%%%

\section{Explicit formulas for \texorpdfstring{$CWR_{4}$}{CWR4} and \texorpdfstring{$CWR_{5}$}{CWR5}}
\label{sec:low-order}

We now specialize the squarefree extraction of
Section~\ref{sec:vertex-variables} to the first components beyond the
previously known cases $k=2,3$ and carry out the resulting
inclusion--exclusion explicitly. For $k=3$ there is nothing to remove.

\begin{proposition}
	\label{prop:k3}
	Let $M$ be the weighted adjacency matrix of a finite simple loopless graph.
	Then every closed $3$-walk is squarefree, and
	$$
	\frac{1}{6}\,\sq\Bigl(\tr\bigl((XM)^{3}\bigr)\Bigr)=\frac{\tr(M^{3})}{6}.
	$$
	Thus, Corollary~\ref{cor:cwrk-sq} for $k=3$ recovers the formula for
	$CWR_{3}$ of Proposition~\ref{prop:cwr23}.
\end{proposition}

\begin{proof}
	A closed walk $i_{1}\to i_{2}\to i_{3}\to i_{1}$ with nonzero weight
	has $i_{1}\ne i_{2}$, $i_{2}\ne i_{3}$ and $i_{3}\ne i_{1}$, because $M$ has
	zero diagonal; hence all closed $3$-walks are already squarefree and $\sq$
	acts on $\tr((XM)^{3})$ by the substitution $x_{i}=1$, giving $\tr(M^{3})$.
\end{proof}

For $k=4$ the walks to be removed are supported on two or three vertices, and
the correction terms can be written in closed matrix form as follows.

	For matrices of the same size, let $M\circ N$ denote the Hadamard
	(entrywise) product. In particular,
	$
	M^{\circ 2}=M\circ M
	$
	is the entrywise square of $M$. Let $\mathbf{1}$ denote the all-ones column
	vector of the appropriate size. For a column vector $y$ with entries in
	$\ZZ[w,r]$, write
	$
	\|y\|_2^2:=y^{\mathsf T}y=\sum_i y_i^2;
	$
	this is only algebraic notation for the indicated quadratic form. For a
	symmetric zero-diagonal matrix $M$, define
	$$
	\Phi_4(M)
	=
	\frac{1}{8}
	\left[
	\tr(M^4)
	-2\left\|M^{\circ 2}\mathbf{1}\right\|_2^2
	+\tr\left((M^{\circ 2})^2\right)
	\right].
	$$

	\begin{theorem}
		\label{thm:cwr4}
		Let $\overline A_B$ and $\overline A_W$ be the weighted adjacency matrices of
		$G_B^*$ and $G_W^*$, respectively. Then
		$$
		CWR_4(L)
		=
		\left(
		\Phi_4(\overline A_B),
		\Phi_4(\overline A_W)
		\right).
		$$
		Equivalently, if $M=(m_{ij})$, then
		$$
		\Phi_4(M)
		=
		\frac{1}{8}
		\left[
		\tr(M^4)
		-2\sum_i\left(\sum_j m_{ij}^2\right)^2
		+\sum_{i,j}m_{ij}^4
		\right].
		$$
	\end{theorem}

\begin{proof}
	It is enough to prove the formula for one of the two weighted Tait graphs, so
	let $M=(m_{ij})$ be the weighted adjacency matrix of a finite simple graph
	without loops. By Theorem~\ref{thm:sq-extraction},
	$
	C_{4}(M)=\frac{1}{8}\,\sq\Bigl(\tr\bigl((XM)^{4}\bigr)\Bigr)
	=\frac{1}{8}\bigl(\tr(M^{4})-D_{4}\bigr),
	$
	where, by Lemma~\ref{lem:walk-expansion}, $\tr(M^{4})$ is the total weight
	of all rooted, oriented closed walks of length $4$ and $D_{4}$ is the total
	weight of those closed $4$-walks that repeat a vertex.

	Since $M$ has zero diagonal, a closed $4$-walk
	$i_{1}\to i_{2}\to i_{3}\to i_{4}\to i_{1}$ with nonzero weight has
	$i_{j}\ne i_{j+1}$ cyclically, so a repetition can occur only through
	$i_{1}=i_{3}$ or $i_{2}=i_{4}$. Put
	$
	S=\sum_i\left(\sum_j m_{ij}^2\right)^2,
	\;
	F=\sum_{i,j}m_{ij}^4.
	$
	The closed $4$-walks supported on two vertices are those with both
	coincidences $i_{1}=i_{3}$ and $i_{2}=i_{4}$; for each ordered pair $(i,j)$
	with $m_{ij}\ne0$ there is exactly one such walk, of weight $m_{ij}^{4}$, so
	together they contribute $F$. The closed $4$-walks supported on exactly
	three vertices are those with exactly one of the two coincidences; each of
	the two cases contributes $S-F$, so together they contribute $2(S-F)$.
	Hence
	$
	D_{4}=F+2(S-F)=2S-F,
	$
	and therefore
	$
	C_4(M)=\frac{1}{8}\bigl(\tr(M^4)-2S+F\bigr).
	$
	Since
	$
	S=\left\|M^{\circ 2}\mathbf{1}\right\|_2^2
	$
	and, because $M$ is symmetric,
	$
	F=\tr\left((M^{\circ 2})^2\right),
	$
	the asserted formula follows. Applying it to $\overline A_B$ and
	$\overline A_W$ gives the two coordinates of $CWR_4(L)$; invariance is
	Theorem~\ref{thm:cwr-invariance}.
\end{proof}

\begin{remark}
	\label{rem:cwr4-walks}
	The correction terms in Theorem~\ref{thm:cwr4} remove precisely the
	length-$4$ closed walks with repeated vertices. Equivalently, since every
	unoriented simple $4$-cycle contributes its edge-weight product exactly
	eight times to $\tr(M^4)$, corresponding to four choices of the initial
	vertex and two orientations (the case $k=4$ of
	Theorem~\ref{thm:sq-extraction}), the proof above amounts to the
	closed-walk identity
	$
	\tr(M^4)=8\,C_4(M)+2S-F.
	$
	For $w=r=1$ the formula reduces to the classical expression
	\cite[Eq.~(2)]{HM71}.
\end{remark}

\begin{example}
	\label{ex:k7a1-again}
	We verify Theorem~\ref{thm:incl-excl} for $k=4$ on the white graph of the knot
	$K7a1$, whose weighted adjacency matrix $M=\overline A_{W}$ is computed in
	\cite[Example~3.6]{JabCWR24}:
	$
	M=
	\begin{bmatrix}
		0 & r^{2} & r^{2} & w\\
		r^{2} & 0 & 0 & w\\
		r^{2} & 0 & 0 & w\\
		w & w & w & 0
	\end{bmatrix}.
	$
	Here $n=4$, so the outer sum of Theorem~\ref{thm:incl-excl} has the single
	term $T=\{1,2,3,4\}$. A direct computation gives
	$
	\tr\bigl(M^{4}\bigr)=8r^{8}+24r^{4}w^{2}+18w^{4},
	$
	$
	\sum_{|S|=3}\tr\bigl(M[S]^{4}\bigr)=12r^{8}+16r^{4}w^{2}+24w^{4},
	\;
	\sum_{|S|=2}\tr\bigl(M[S]^{4}\bigr)=4r^{8}+6w^{4},
	$
	while the subsets with $|S|\le1$ contribute $0$. Hence
	$
	\gamma(T)=\bigl(8r^{8}+24r^{4}w^{2}+18w^{4}\bigr)
	-\bigl(12r^{8}+16r^{4}w^{2}+24w^{4}\bigr)
	+\bigl(4r^{8}+6w^{4}\bigr)
	=8r^{4}w^{2},
	$
	and $C_{4}(M)=\gamma(T)/8=r^{4}w^{2}$. This is the total weight of the unique
	simple $4$-cycle of $G_{W}^{*}$, and it agrees with the white coordinate of
	$CWR_{4}(K7a1)=(2r^{2}w^{2},r^{4}w^{2})$ listed in
	\cite[Table~1]{JabCWR24}. Equivalently, by Theorem~\ref{thm:logdet-main},
	$
	\bigl[t^{4}\bigr]\,\sq\Bigl(-\log\det\bigl(I-tXM\bigr)\Bigr)=2r^{4}w^{2}.
	$
	The same value is obtained from Theorem~\ref{thm:cwr4}: here
	$\tr(M^{4})=8r^{8}+24r^{4}w^{2}+18w^{4}$, $S=6r^{8}+8r^{4}w^{2}+12w^{4}$ and
	$F=4r^{8}+6w^{4}$, so that $\Phi_{4}(M)=\tfrac18\cdot 8r^{4}w^{2}=r^{4}w^{2}$.
\end{example}

The squarefree extraction of Section~\ref{sec:vertex-variables} makes the next
component computable in closed form as well. Recall the Hadamard powers
$M^{\circ 2}$ and $M^{\circ 3}$ and the all-ones vector $\onev$, and for a
symmetric zero-diagonal matrix $M=(m_{ij})$ put
$
\Delta(M)=\operatorname{diag}\bigl(M^{\circ 2}\onev\bigr),
\;\text{that is,}\;
\Delta(M)_{ii}=\sum_{j}m_{ij}^{2},
$
the diagonal matrix of weighted square-degrees. Define
$
\Phi_{5}(M)=\frac{1}{10}
\Bigl[
\tr(M^{5})-5\,\tr\bigl(\Delta(M)\,M^{3}\bigr)+5\,\tr\bigl(M^{\circ 3}M^{2}\bigr)
\Bigr].
$

\begin{theorem}
	\label{thm:cwr5}
	Let $\overline A_{B}$ and $\overline A_{W}$ be the weighted adjacency matrices of
	$G_{B}^{*}$ and $G_{W}^{*}$. Then
	$
	CWR_{5}(L)=\bigl(\Phi_{5}(\overline A_{B}),\,\Phi_{5}(\overline A_{W})\bigr).
	$
	Equivalently, in entries,
	$$
	\Phi_{5}(M)=\frac{1}{10}
	\left[
	\tr(M^{5})
	-5\sum_{i}\Bigl(\sum_{j}m_{ij}^{2}\Bigr)(M^{3})_{ii}
	+5\sum_{i,j}m_{ij}^{3}(M^{2})_{ij}
	\right].
	$$
\end{theorem}

\begin{proof}
	As in Theorem~\ref{thm:cwr4} it suffices to treat one weighted Tait graph, so
	let $M$ be the weighted adjacency matrix of a finite simple loopless graph $G$.
	By Theorem~\ref{thm:sq-extraction},
	$
	C_{5}(M)=\frac{1}{10}\,\sq\Bigl(\tr\bigl((XM)^{5}\bigr)\Bigr)
	=\frac{1}{10}\bigl(\tr(M^{5})-D_{5}\bigr),
	$
	where $D_{5}$ is the total weight of the closed $5$-walks that repeat a vertex.
	
	\emph{Step 1: the shape of a degenerate closed $5$-walk.}
	Let $W$ be a closed $5$-walk and let $H$ be the multigraph formed by its five
	edge-steps, so that $W$ is an Eulerian circuit of $H$; in particular $H$ is
	connected, loopless, and all its degrees are even, with degree sum $10$. Let
	$s$ be the number of vertices of $H$. A closed walk alternating between two
	vertices has even length, so $s\ge 3$, and $s=5$ exactly when $W$ is a
	traversal of a simple $5$-cycle.
	
	Suppose $s=4$. The even degrees sum to $10$, hence are $(4,2,2,2)$; let $a$ be
	the vertex of degree $4$ and $b,c,d$ the remaining ones. Since $H$ is
	loopless, exactly four of the five edges of $H$, counted with multiplicity, are
	incident to $a$, so exactly one edge of $H$ lies inside $\{b,c,d\}$; say it
	joins $b$ and $c$. Then $b$ and $c$ each have one remaining degree-slot, which
	must be filled by an edge to $a$, and both remaining slots at $d$ are filled by
	edges to $a$. Hence, $H$ is the triangle $abc$ together with the edge $ad$
	taken twice.
	
	Suppose $s=3$, with vertices $a,b,c$ and edge multiplicities
	$m_{ab},m_{ac},m_{bc}$. The even degrees summing to $10$ are $(4,4,2)$, the
	sequence $(6,2,2)$ being impossible because the corresponding linear system
	forces the multiplicity $m_{bc}=-1$. Solving the system
	$m_{ab}+m_{ac}=4$, $m_{ab}+m_{bc}=4$, $m_{ac}+m_{bc}=2$ gives
	$(m_{ab},m_{ac},m_{bc})=(3,1,1)$, so $H$ is the triangle $abc$ with the edge
	$ab$ taken three times, that is, the triangle together with one of its own
	edges taken twice more.
	
	Thus, in both cases every degenerate closed $5$-walk is a triangle traversal
	with an out-and-back excursion along an edge $\{a,d\}$ incident to one of its
	vertices: for $s=4$ the vertex $d$ lies outside the triangle, and for $s=3$ the
	edge $\{a,d\}$ is an edge of the triangle. This is the case
	$n=5$ of the shape classification of \cite{HM71}; the two shapes are those
	displayed in \cite[Fig.~3]{HM71}.
	
	\emph{Step 2: counting.}
	Fix a triangle $T$ with vertex set $\{a,b,c\}$, a vertex $a$ of $T$ and a
	neighbour $d$ of $a$ in $G$. The corresponding closed walks are obtained by
	concatenating, at $a$, the excursion $a\to d\to a$ with a traversal of $T$;
	there are two choices of orientation of $T$ and five choices of the initial
	step, giving ten sequences, each of weight $\wgt(T)\,m_{ad}^{2}$.
	
	These ten sequences are pairwise distinct. Two rotations of one and the same
	circuit agree only if the circuit is invariant under a nontrivial rotation;
	since $5$ is prime, such a rotation generates the full cyclic group and forces
	all five entries to coincide, which is impossible as $a\ne b$. It remains to
	check that no rotation of the reversed circuit equals a rotation of the
	original one. If $d\in V(T)$, relabel the two vertices of
	$T\setminus\{a\}$ so that $d=b$; thus in all cases considered below
	we may assume $d\neq c$. Writing the original circuit as $(a,d,a,b,c)$, its reversal is
	$(c,b,a,d,a)$. The vertex $c$ occurs exactly once in the original circuit, so
	the only rotation of the original beginning with $c$ is $(c,a,d,a,b)$;
	comparing second entries with $(c,b,a,d,a)$ gives $b=a$, a contradiction.
	
	Distinct data $(T,a,d)$ give distinct walks, with one exception: if $d$ is a
	vertex of $T$, then $(T,a,d)$ and $(T,d,a)$ produce the same ten walks, since
	the two resulting circuits differ only by a rotation or a reversal. Hence
	$
	D_{5}=10\left[
	\sum_{T}\wgt(T)\sum_{a\in V(T)}\ \sum_{d}m_{ad}^{2}
	\;-\;\sum_{T}\wgt(T)\sum_{e\in E(T)}\wgt(e)^{2}
	\right],
	$
	both outer sums being over the triangles $T$ of $G$.
	
	\emph{Step 3: matrix form.}
	The inner sum $\sum_{d}m_{ad}^{2}$ is $\Delta(M)_{aa}$, and
	$(M^{3})_{aa}=2\sum_{T\ni a}\wgt(T)$, because the closed $3$-walks based at
	$a$ are the two orientations of the triangles through $a$, so
	$
	\sum_{T}\wgt(T)\sum_{a\in V(T)}\Delta(M)_{aa}
	=\frac{1}{2}\sum_{a}\Delta(M)_{aa}(M^{3})_{aa}
	=\frac{1}{2}\tr\bigl(\Delta(M)M^{3}\bigr).
	$
	For the second sum, expand
	$
	\tr\bigl(M^{\circ 3}M^{2}\bigr)=\sum_{i,j,k}m_{ij}^{3}m_{jk}m_{ki}.
	$
	Since $M$ has zero diagonal, only triples of pairwise distinct indices
	contribute, i.e. only triangles $\{i,j,k\}$ with a distinguished ordered edge
	$(i,j)$; each triangle and each of its three edges occurs for two orderings,
	so
	$
	\tr\bigl(M^{\circ 3}M^{2}\bigr)=2\sum_{T}\wgt(T)\sum_{e\in E(T)}\wgt(e)^{2}.
	$
	Therefore
	$
	D_{5}=5\Bigl[\tr\bigl(\Delta(M)M^{3}\bigr)-\tr\bigl(M^{\circ 3}M^{2}\bigr)\Bigr],
	$
	and substituting this into $C_{5}(M)=\bigl(\tr(M^{5})-D_{5}\bigr)/10$ gives
	$C_{5}(M)=\Phi_{5}(M)$. Applying this to $\overline A_{B}$ and
	$\overline A_{W}$ yields the two coordinates of $CWR_{5}(L)$.
\end{proof}

\begin{remark}
	\label{rem:cwr5-unweighted}
	If all edge weights are equal to $1$, then $\Delta(M)$ is the diagonal matrix
	of vertex degrees and $M^{\circ 3}=M$, so Theorem~\ref{thm:cwr5} reduces to
	$
	\frac{1}{10}\Bigl[\tr(A^{5})-5\sum_{i}d_{i}(A^{3})_{ii}+5\tr(A^{3})\Bigr],
	$
	the classical count of the $5$-cycles of a graph, which is
	\cite[Eq.~(3)]{HM71}.
\end{remark}

\begin{example}
	\label{ex:k7a1-cwr5}
	For the knot $K7a1$ the black graph $G_{B}^{*}$ has the weighted adjacency
	matrix
	$
	\overline A_{B}=
	\begin{bmatrix}
		0 & 0 & r & r & 0\\
		0 & 0 & 0 & r & r\\
		r & 0 & 0 & w & w\\
		r & r & w & 0 & w\\
		0 & r & w & w & 0
	\end{bmatrix}
	$
	of \cite[Example~3.6]{JabCWR24}. The three terms of $\Phi_{5}$ are
	$
	\tr(\overline A_{B}^{5})=70r^{4}w+100r^{2}w^{3}+30w^{5},
	$
	$
	\tr\bigl(\Delta(\overline A_{B})\overline A_{B}^{3}\bigr)
	=20r^{4}w+24r^{2}w^{3}+12w^{5},
	\;
	\tr\bigl(\overline A_{B}^{\circ 3}\overline A_{B}^{2}\bigr)
	=8r^{4}w+4r^{2}w^{3}+6w^{5},
	$
	so that
	$
	\Phi_{5}(\overline A_{B})
	=\tfrac{1}{10}\bigl[(70-100+40)r^{4}w+(100-120+20)r^{2}w^{3}
	+(30-60+30)w^{5}\bigr]=r^{4}w .
	$
	This is confirmed directly: the vertices $1$ and $2$ have degree $2$, so any
	Hamiltonian cycle must use all four of their incident edges, and one checks
	that the only simple $5$-cycle of $G_{B}^{*}$ is $3\to1\to4\to2\to5\to3$, of
	weight $r^{4}w$. The graph $G_{W}^{*}$ has only four vertices, so
	$\Phi_{5}(\overline A_{W})=0$. This agrees with
	$CWR_{5}(K7a1)=(r^{4}w,0)$ in \cite[Table~1]{JabCWR24}.
\end{example}

%%%%%%%%%%%%

\section{Odd components, bipartiteness, and the characteristic polynomial}
\label{sec:odd-cycles}

The formulas above compute one component of $CWR$ at a time. For the odd
components there is, in addition, a purely spectral criterion for the vanishing
of all of them at once, and a closed expression for the first one that does not
vanish. Throughout this section $G$ denotes one of the consolidated Tait
graphs $G_{B}^{*}$ or $G_{W}^{*}$, with weighted adjacency matrix $\overline A$
and ordinary adjacency matrix $A=A(G)$; by $C_{k}(\overline A)$ we mean, as
above, the total weight of the simple $k$-cycles of $G$, that is, the black or
white coordinate of $CWR_{k}$ according to the choice of $G$. The statements
are again purely graph-theoretic, and are transferred to the invariant by
Theorem~\ref{thm:cwr-invariance}.

\begin{proposition}
	\label{prop:cwr-bipartite}
	Let $G$ be either of the consolidated Tait graphs
	$G_B^{*}$ or $G_W^{*}$, let $A=A(G)$ be its ordinary adjacency
	matrix, and write
	$
	\det(\lambda I-A)
	=
	\lambda^n+c_1\lambda^{n-1}+\cdots+c_n.
	$
	Then the following conditions are equivalent:
	\begin{enumerate}
		\item the corresponding component of $CWR_k$ vanishes for every odd
		$k\ge3$;
		\item $G$ contains no odd cycle;
		\item $G$ is bipartite;
		\item
		$
		c_{2j+1}=0
		$
		for every $j$ for which the coefficient is defined;
		\item the adjacency spectrum of $G$ is symmetric about zero:
		if $\lambda$ is an eigenvalue of $A$ with multiplicity $m$, then
		$-\lambda$ is also an eigenvalue with multiplicity $m$.
	\end{enumerate}
\end{proposition}

\begin{proof}
	For $k\ge3$, the corresponding component of $CWR_k$ is a sum over
	the simple $k$-cycles of $G$ of monomials with positive integer
	coefficients. Since no cancellation between such monomials is possible, it
	vanishes if and only if $G$ contains no $k$-cycle. Thus, (1) is equivalent to
	the absence of odd cycles, and hence to bipartiteness.
	
	The equivalence of bipartiteness and symmetry of the adjacency spectrum
	about the origin is classical; see
	\cite[Theorems~5.2.2--5.2.3 and Corollary~5.2.2]{WW19}
	and \cite[Theorem~3.14]{Bap14}. The corresponding parity vanishing
	of the characteristic-polynomial coefficients also follows directly
	from the Sachs expansion; see \cite[Corollary~5.2.7]{WW19}.
	
	Moreover, condition~(4) is equivalent to
	$
	\det(-\lambda I-A)=(-1)^n\det(\lambda I-A),
	$
	and hence to symmetry of the multiset of roots of the characteristic
	polynomial about zero. Thus, (4) is equivalent to~(5).
\end{proof}

\begin{corollary}
	\label{cor:first-odd-cwr}
	Let $G$ be one of $G_B^{*},G_W^{*}$, with weighted adjacency matrix
	$\overline A$ and ordinary adjacency matrix $A=A(G)$, let
	$
	\det(\lambda I-A)
	=
	\lambda^n+c_1\lambda^{n-1}+\cdots+c_n,
	$
	and suppose $2q+1\le n$. If the corresponding components of
	$
	CWR_3,CWR_5,\ldots,CWR_{2q-1}
	$
	all vanish, then
	$$
	\left.
	C_{2q+1}(\overline A)
	\right|_{w=r=1}
	=
	-\frac{1}{2}c_{2q+1}.
	$$
	
	Equivalently, when $2q+1$ is the length of the shortest odd cycle
	of $G$, the number of such shortest odd cycles is
	$-c_{2q+1}/2$.
\end{corollary}

\begin{proof}
	The assumed vanishing means that $G$ contains no odd cycle of length
	less than $2q+1$. Apply Sachs' expansion of the adjacency
	characteristic polynomial \cite[Theorem~5.2.9]{WW19}; compare also
	\cite[Corollary~3.11]{Bap14}. A Sachs
	subgraph on $2q+1$ vertices is a disjoint union of single edges and
	simple cycles. Since its number of vertices is odd, at least one of
	its cycle components must have odd length. By the hypothesis such a
	cycle has length at least $2q+1$, and therefore it must use all
	$2q+1$ vertices. Consequently the Sachs subgraphs contributing to
	the coefficient $c_{2q+1}$ are precisely the simple $(2q+1)$-cycles.
	
	Each such cycle has one component and one cyclic component, hence
	contributes
	$
	(-1)^1 2^1=-2
	$
	to Sachs' formula. Therefore
	$
	c_{2q+1}
	=
	-2\,\#\mathcal C_{2q+1}(G).
	$
	After setting $w=r=1$, every cycle weight is equal to $1$, so
	$
	\left.C_{2q+1}(\overline A)\right|_{w=r=1}
	=
	\#\mathcal C_{2q+1}(G)
	=
	-\frac12c_{2q+1},
	$
	as claimed.
\end{proof}

\begin{remark}
	\label{rem:q1}
	For $q=1$ the list $CWR_{3},\ldots,CWR_{2q-1}$ is empty, so the hypothesis of
	Corollary~\ref{cor:first-odd-cwr} is vacuous and the statement reduces to the
	classical identity $c_{3}=-2\,c_{3}(G)$, where $c_{3}(G)$ is the number of
	triangles; compare the trace formula $c_{3}(G)=\tr(A^{3})/6$ recalled after
	Proposition~\ref{prop:cwr23}.
\end{remark}

%%%%%%%%%%%%
\section{Conclusion}

The weighted adjacency matrices of the consolidated Tait graphs contain the
entire $CWR$ sequence once repeated-vertex closed walks are removed. Vertex
variables make this removal uniform: squarefree extraction gives $CWR_k$ for
every $k\ge3$, the trace--log packages all lengths into a finite polynomial, and
Boolean M\"obius inversion converts the same information into traces of
principal submatrices. The explicit formulas for $CWR_4$ and $CWR_5$ are the
first low-order instances beyond the previously known cases $k=2,3$, and the
odd components admit in addition the spectral description of
Section~\ref{sec:odd-cycles}.

The construction also separates the graph-theoretic and knot-theoretic parts of
the argument. The matrix identities hold for arbitrary finite simple loopless
weighted graphs, whereas invariance under changing a reduced alternating
diagram is supplied by the existing $CWR$ invariance theorem. This makes the
formulas suitable both for symbolic computation of the invariant and for
comparison with general exact algorithms for weighted simple-cycle counting.

\section*{Acknowledgements}

The author used OpenAI's ChatGPT for editorial assistance and for exploring formulations.


\begin{thebibliography}{99}

\bibitem{AKT21}
C.~Adams, E.~Flapan, A.~Henrich, L.~H.~Kauffman,
L.~D.~Ludwig and S.~Nelson (eds.),
\emph{Encyclopedia of Knot Theory},
CRC Press, 2021.

\bibitem{Bap14}
R.~B.~Bapat,
\emph{Graphs and Matrices},
2nd ed.,
Universitext,
Springer, London, 2014,
\url{https://doi.org/10.1007/978-1-4471-6569-9}.

\bibitem{Die17}
R.~Diestel,
\emph{Graph Theory},
5th ed.,
Graduate Texts in Mathematics, vol.~173,
Springer, Berlin, 2017,
\url{https://doi.org/10.1007/978-3-662-53622-3}.

\bibitem{DMSY21}
A.~Dochtermann, E.~Meyers, R.~Samavedam and A.~Yi,
\emph{Integral flow and cycle chip-firing on graphs},
arXiv:2006.13397v3, 2021.

\bibitem{FS09}
P.~Flajolet and R.~Sedgewick,
\emph{Analytic Combinatorics},
Cambridge University Press, Cambridge, 2009.

\bibitem{GKW19}
P.-L.~Giscard, N.~Kriege and R.~C.~Wilson,
\emph{A general purpose algorithm for counting simple cycles and simple paths
	of any length},
Algorithmica \textbf{81} (2019), no.~7, 2716--2737,
\url{https://doi.org/10.1007/s00453-019-00552-1}.

\bibitem{GRW18}
P.-L.~Giscard, P.~Rochet and R.~C.~Wilson,
\emph{A Hopf algebra for counting cycles},
Discrete Math. \textbf{341} (2018), no.~5, 1439--1448,
\url{https://doi.org/10.1016/j.disc.2017.10.002}.

\bibitem{Gre17}
J.~E.~Greene,
\emph{Alternating links and definite surfaces},
Duke Math. J. \textbf{166} (2017), no.~11, 2133--2151.

\bibitem{HM71}
F.~Harary and B.~Manvel,
\emph{On the number of cycles in a graph},
Mat. \v{C}asopis Sloven. Akad. Vied \textbf{21} (1971), 55--63.

\bibitem{JabWRP24}
M.~Jab{\l}onowski,
\emph{A polynomial pair invariant of alternating knots and links},
J. Knot Theory Ramifications \textbf{33} (2024), no.~14, 2450048,
\url{https://doi.org/10.1142/S0218216524500482}.

\bibitem{JabCWR24}
M.~Jab{\l}onowski,
\emph{CWR sequence of invariants of alternating links and its properties},
J. Knot Theory Ramifications \textbf{33} (2024), no.~14, 2450052,
\url{https://doi.org/10.1142/S0218216524500524}.

\bibitem{Kin22}
T.~Kindred,
\emph{A geometric proof of the flyping theorem},
arXiv:2008.06490v2, 2022.

\bibitem{KK19}
U.~Knauer and K.~Knauer,
\emph{Algebraic Graph Theory: Morphisms, Monoids and Matrices},
2nd ed., De Gruyter Studies in Mathematics, vol.~41,
De Gruyter, Berlin/Boston, 2019.

\bibitem{LW00}
M.~Lien and W.~Watkins,
\emph{Dual graphs and knot invariants},
Linear Algebra Appl. \textbf{306} (2000), 123--130.

\bibitem{MP93}
K.~Murasugi and J.~H.~Przytycki,
\emph{An index of a graph with applications to knot theory},
Mem. Amer. Math. Soc. \textbf{106} (1993), no.~508.

\bibitem{Prz06}
J.~H.~Przytycki,
\emph{Knots: From combinatorics of knot diagrams to combinatorial topology
	based on knots},
arXiv:math/0601227v1, 2006,
Chapter~V: \emph{Graphs and links}.

\bibitem{SS08}
R.~Schott and G.~S.~Staples,
\emph{Nilpotent adjacency matrices, random graphs, and quantum random
variables},
J. Phys. A: Math. Theor. \textbf{41} (2008), no.~15, 155205,
\url{https://doi.org/10.1088/1751-8113/41/15/155205}.

\bibitem{SS11}
R.~Schott and G.~S.~Staples,
\emph{Complexity of counting cycles using zeons},
Comput. Math. Appl. \textbf{62} (2011), no.~4, 1828--1837,
\url{https://doi.org/10.1016/j.camwa.2011.06.026}.


\bibitem{SW19}
D.~S.~Silver and S.~G.~Williams,
\emph{Knot invariants from Laplacian matrices},
J. Knot Theory Ramifications \textbf{28} (2019), 1950058.

\bibitem{Sin26}
T.~Sinclair,
\emph{Finite free convolution via reproducing kernels and squarefree algebras},
arXiv:2606.10870v2, 2026,
\url{https://doi.org/10.48550/arXiv.2606.10870}.

\bibitem{Sta08}
G.~S.~Staples,
\emph{A new adjacency matrix for finite graphs},
Adv. Appl. Clifford Algebras \textbf{18} (2008), 979--991,
\url{https://doi.org/10.1007/s00006-008-0116-5}.

\bibitem{Thi87}
M.~B.~Thistlethwaite,
\emph{A spanning tree expansion of the Jones polynomial},
Topology \textbf{26} (1987), no.~3, 297--309.

\bibitem{Tra04}
P.~Traczyk,
\emph{A combinatorial formula for the signature of alternating diagrams},
Fund. Math. \textbf{184} (2004), 311--316.

\bibitem{WW19}
S.~Wagner and H.~Wang,
\emph{Introduction to Chemical Graph Theory},
Discrete Mathematics and Its Applications,
CRC Press, Boca Raton, 2019.

\bibitem{Yad23}
S.~K.~Yadav,
\emph{Advanced Graph Theory},
Springer, Cham, 2023,
\url{https://doi.org/10.1007/978-3-031-22562-8}.

\end{thebibliography}
\end{document}